\documentclass[12pt,bezier]{article}
\usepackage{times}
\usepackage{booktabs}
\usepackage{pifont}
\usepackage{floatrow}
\usepackage{caption}
\usepackage{mathrsfs}
\usepackage[fleqn]{amsmath}
\usepackage{amsfonts,amsthm,amssymb,mathrsfs,bbding}
\usepackage{txfonts}
\usepackage{graphics,multicol}
\usepackage{graphicx}
\usepackage{color}
\usepackage{caption}
\usepackage{cite}
\usepackage{latexsym,bm}
\usepackage{indentfirst}
\usepackage{color}
\usepackage[colorlinks=true,anchorcolor=blue,filecolor=blue,linkcolor=blue,urlcolor=blue,citecolor=blue]{hyperref}
\usepackage{extarrows}
\usepackage{cite}
\usepackage{latexsym,bm}
\usepackage{mathtools}
\evensidemargin=\oddsidemargin\topmargin=-1.5cm

\newtheorem{thm}{Theorem}[section]
\newtheorem{prob}{Problem}[section]

\newtheorem{lem}{Lemma}[section]
\newtheorem{cor}{Corollary}[section]

\newtheorem{fact}{Fact}[section]

\newtheorem{claim}{Claim}[section]
\theoremstyle{definition}

\addtocounter{section}{0}

\begin{document}
\title{ A strengthening of the spectral chromatic critical edge theorem: books and theta graphs
\footnote{Supported by the National Natural Science Foundation of China (Nos. 12171066, 11771141, 12011530064)
and Anhui Provincial Natural Science Foundation (No. 2108085MA13).}}
\author{{\bf Mingqing Zhai$^{a}$}
, {\bf Huiqiu Lin$^{b}$}\thanks{Corresponding author. E-mail addresses: mqzhai@chzu.edu.cn
(M. Zhai); huiqiulin@126.com (H. Lin).}
\\
{\footnotesize $^a$ School of Mathematics and Finance, Chuzhou University, Chuzhou, Anhui 239012, China} \\
{\footnotesize $^b$ School of Mathematics, East China University of Science and Technology, Shanghai 200237, China}}
\date{}

\date{}
\maketitle
{\flushleft\large\bf Abstract}
The chromatic critical edge theorem of Simonovits states that for a given color critical graph
$H$ with $\chi(H)=k+1$, there exists an $n_0(H)$ such that the Tur\'an graph $T_{n,k}$
is the only extremal graph with respect to $ex(n,H)$ provided $n \geq n_0(H)$.
Nikiforov's pioneer work on spectral graph theory implies that the color critical edge theorem
also holds if $ex(n,H)$ is replaced by the maximum spectral radius and $n_0(H)$ is an exponential function of $|H|$.
We want to know which color critical graphs
$H$ satisfy that $n_0(H)$ is a linear function of $|H|$.
Previous graphs include complete graphs and odd cycles. In this
paper, we find two new classes of graphs: books and theta
graphs. Namely, we prove that
every graph on $n$ vertices with $\rho(G)>\rho(T_{n,2})$ contains a
book of size greater than $\frac{n}{6.5}$.
This can be seen as a spectral version of a 1962 conjecture by Erd\H{o}s,
which states that every graph on $n$ vertices with $e(G)>e(T_{n,2})$ contains a
book of size greater than $\frac{n}{6}$.
In addition, our result on theta graphs
implies that if $G$ is a graph of order $n$ with $\rho(G)>\rho(T_{n,2})$, then
$G$ contains a cycle of length $t$ for every $t\leq \frac{n}{7}$.
This is related to an open question by Nikiforov which asks to determine the maximum
$c$ such that every graph $G$ of large enough
order $n$ with $\rho(G)>\rho(T_{n,2})$ contains a cycle
of length $t$ for every $t\leq cn$.

\begin{flushleft}
\end{flushleft}

\section{Introduction}

The \emph{Tur\'{a}n number} $ex(n,H)$ is the maximum number of edges
over all graphs with $n$ vertices that do not contain $H$ as a subgraph.
A graph on $n$ vertices is said to be \emph{extremal} with respect to $ex(n,H)$,
if it has exactly $ex(n,H)$ edges and does not contain a subgraph isomorphic to $H$.
The \emph{Tur\'{a}n graph},
denoted by $T_{n,k}$, is the complete $k$-partite graph $K_{n_1,n_2,\ldots,n_k}$,
where $\sum_{i=1}^k n_i=n$ and $\lfloor\frac{n}{k}\rfloor\le n_i \le \lceil\frac{n}{k}\rceil$.
A main motivation for studying Tur\'{a}n numbers is that they are often useful in Ramsey theory, where the original
statements can be seen in \cite{D}.
There are very few cases when the exact value of $ex(n,H)$ is known.
By Tur\'{a}n's Theorem \cite{bb}, $ex(n, K_{k+1}) = e(T_{n,k})$ for $n\geq k+1$,
and $T_{n,k}$ is the unique extremal graph.
The Tur\'{a}n number $ex(n,H)$ is determined for some cases when $H$ is acyclic, see $ex(n,P_k)$ \cite{ER,bah,ba1},
$ex(n, k\cdot P_2)$ \cite{ER}, $ex(n, k\cdot P_3)$ \cite{BK,LT}, $ex(n,k\cdot K_{1,k})$ \cite{LAN2}.
More extremal results on forests can be seen in \cite{BHC,YWH}.
F\"{u}redi and Gunderson \cite{aaa} determined $ex(n, C_{2k+1})$ for all $k$ and $n$ and characterized all the extremal graphs.
In particular, $ex(n, C_{2k+1})=e(T_{n,2})=\lfloor\frac{n^2}{4}\rfloor$ for $n\geq4k-2$.

Let $A(G)$ be the \emph{adjacency matrix} of a graph $G$. The largest modulus
of all eigenvalues of $A(G)$ is called \emph{the spectral radius} of $G$ and denoted by $\rho(G)$.
Nikiforov \cite{V5} posed a \emph{Brualdi-Solheid-Tur\'{a}n type problem}: what is the
maximal spectral radius of an $H$-free graph with $n$ vertices?
Let $$ex_{sp}(n,H)=\max\{\rho(G):|G|=n, H\nsubseteq G\}.$$
A graph $G$ is said to be
\emph{extremal} with respect to $ex_{sp}(n,H)$, if $\rho(G)=ex_{sp}(n,H)$ and $G$ is an $H$-free graph of order $n$.
We use $Ex_{sp}(n,H)$ to denote the set of extremal graphs with respect to $ex_{sp}(n,H)$.

In the past decade, much attention has been paid to the above extremal problem,
see $ex_{sp}(n,K_{k+1})$ \cite{V1,WILF}, $ex_{sp}(n,K_{s,t})$ \cite{BG,V1,Niki}, $ex_{sp}(n,P_k)$ \cite{V5},
$ex_{sp}(n,\bigcup_{i=1}^kP_{a_i})$ \cite{CLZ}, $ex_{sp}(n,\bigcup_{i=1}^kS_{d_i})$ \cite{CLZ2},
$ex_{sp}(n,C_4)$ \cite{V1,Zhai}, $ex_{sp}(n,C_6)$ \cite{ZL}, $ex_{sp}(n,$ $C_{2k+1})$ \cite{V5}, $ex_{sp}(n,F_k)$ \cite{CFTZ}
(where $F_k$ is the friendship graph), (outer) planar graphs \cite{LIN,TT}.
For more results on extremal spectral graph theory, see \cite{RON,LNW,LU,0}.

An edge $e$ in a graph $H$ is called a \emph{color critical edge}, if $\chi(H-{e})=\chi(H)-1$,
where $\chi(H)$ is the chromatic number of $H$.
The following result is called \emph {chromatic critical edge theorem} by F\"{u}redi and Gunderson \cite{aaa}.

\begin{thm}\label{tm00} (Simonovits,\cite{aa})
Let $H$ be a graph containing a color critical edge and $\chi(H)=k+1$.
Then there exists an $n_0(H)$ such that if $n\geq n_0(H)$ then $T_{n,k}$ is the only extremal graph with respect to $ex(n,H)$.
\end{thm}

Having seen various spectral forms of Tur\'{a}n-type results, one can expect that more
results that surround it can be cast in spectral form as well. This is indeed, for example,
spectral Erd\H{o}s-Stone-Bollob\'{a}s theorem \cite{ESB} and
spectral version of saturation problem \cite{ESB1}.
By Nikiforov's result (see Theorem 2,\cite{ESB1}) and a more careful calculation on the equality case in his proof,
one can get the following spectral version of chromatic critical edge theorem.

\begin{thm}\label{tm01}
Let $H$ be a graph containing a color critical edge and $\chi(H)=k+1$.
Then there exists an $n_0(H)$ (where $n_0(H)\geq e^{|H|k^{(2k+9)(k+1)}}$)
such that if $n\geq n_0(H)$ then $T_{n,k}$ is the only extremal graph with respect to $ex_{sp}(n,H)$.
\end{thm}

We can observe that if $n\geq k$ and $T_{n,k}$ is an extremal graph with respect to $ex_{sp}(n,H)$
then it is also an extremal graph with respect to $ex(n,H)$.
Indeed, set $s=\lfloor\frac{n} {k}\rfloor$.
Let $T_{n,k}$ be an extremal graph with respect to $ex_{sp}(n,H)$,
and $G^*$ be an extremal graph with respect to $ex(n,H)$.
Note that $\rho(T_{n,k})=\frac12\left(n-2s-1+\sqrt{(n-2s-1)^2+4s(s+1)(k-1)}\right).$
By a well-known inequality $\rho(G)\geq \frac {2e(G)}{n}$ due to Collatz and Sinogowitz \cite{CO},
we have %$$e(G^*)\leq\frac{n}{2}\rho(G^*)\leq\frac{n}{2}\rho(T_{n,k})=\frac{n(n-2s-1)+n\sqrt{(n-2s-1)^2+4s(s+1)(k-1)}}4.$$
$$e(G^*)\leq\frac{n}{2}\rho(G^*)\leq\frac{n}{2}\rho(T_{n,k})=\frac14\left(n(n-2s-1)+n\sqrt{(n-2s-1)^2+4s(s+1)(k-1)}\right).$$
On the other hand, $$e(T_{n,k})=\frac12\sum d(v)=\frac12\left(n(n-2s-1)+s(s+1)k\right).$$
By a direct computation, one can find that $e(G^*)\leq\frac{n}{2}\rho(T_{n,k})<e(T_{n,k})+1$ for $n\geq k$.
It follows that $e(G^*)\leq e(T_{n,k})$ and $T_{n,k}$ is also an extremal graph with respect to $ex(n,H)$.
Thus we have the following result.

\begin{fact}\label{fact1}
Theorem \ref{tm01} can imply Theorem \ref{tm00}.
\end{fact}

The main interest of the paper is the following problem.
We believe that this problem is very useful to study the existence of desired subgraphs.

\begin{prob}\label{que0}
Which color critical graphs $H$ satisfy
that $n_0(H)$ is a linear function of $|H|$, where
$n_0(H)$ is defined in Theorem \ref{tm01}?
\end{prob}

Up to now, Problem \ref{que0} holds for complete graphs \cite{V1} and odd cycles \cite{NI2}.
More precisely, $n_0(K_r)=r$ and $n_0(C_{2r-1})\leq640r$.
In order to find more support on Problem \ref{que0}, we show the following two theorems. The proofs
are given in Sections \ref{se3} and \ref{se4}.

Firstly, let $B_{r+1}$ be the \emph{$(r+1)$-book},
that is, the graph obtained from $K_{2,r+1}$ by adding an edge
within the partition set of two vertices.
Obviously, $B_{r+1}$ contains a color critical edge and $\chi(B_{r+1})=3$.

\begin{thm}\label{thm1}
$Ex_{sp}(n,B_{r+1})=\{T_{n,2}\}$ for $n\geq \frac{13}{2}r.$
\end{thm}

The size of a book $B_r$ is the number $r$ of triangles in it.
The size of the largest book in a
graph $G$ is called the \emph{booksize} of $G$.
In 1962, Erd\H{o}s \cite{Er} initiated the study of books in graphs.
Since then, books have attracted considerable attention in extremal graph theory (see, for example, \cite{BB, Er1,Er2,KH}).
%and Ramsey theory (see, for example, \cite{LL,NI3,NI4,NI5,CC}).
Erd\H{o}s \cite{Er} proposed the following conjecture:
the booksize of a graph $G$ on $n$ vertices with $e(G)>e(T_{n,2})$ is greater than $\frac{n}6$.
This conjecture was proved by Edwards in an unpublished manuscript \cite{ED} and independently by Khad\v{z}iivanov and Nikiforov in \cite{KH}.
Moreover, they constructed a graph with $n=6r$ vertices, $m>e(T_{n,2})$ edges
such that its booksize is $\frac n6+1$. This implies that the booksize $\frac{n}6$ is best possible.
Since we focus on spectral extremal problems,
it is natural to ask what is the minimum booksize of a graph $G$ on $n$ vertices with $\rho(G)>\rho(T_{n,2})$?
Theorem \ref{thm1} gives the following answer.

\begin{cor}\label{cor1}
For arbitrary positive integer $n$,
the booksize of a graph $G$ on $n$ vertices with $\rho(G)\geq\rho(T_{n,2})$ is greater than $\frac{n}{6.5}$,
unless $G\cong T_{n,2}$.
\end{cor}

Recall that $\frac n2\rho(T_{n,k})<e(T_{n,k})+1$.
If $e(G)>e(T_{n,2})$, then $$\rho(G)\geq \frac{2e(G)}n\geq\frac{2(e(T_{n,2})+1)}n>\rho(T_{n,2}).$$
Thus, we propose a stronger problem than Erd\H{o}s' conjecture for further research.

\begin{prob}\label{con3}
For arbitrary positive integer $n$, if $G$ is a graph of
order $n$ with $\rho(G)>\rho(T_{n,2})$, is it true that the booksize of $G$ is greater than $\frac{n}6$?
\end{prob}

A \emph{generalized theta graph} $\theta(l_1,l_2,\ldots,l_t)$
is the graph obtained by connecting two vertices with $t$ internally disjoint paths of lengths $l_1,l_2,\ldots,l_t$,
where $l_1\leq l_2\ldots\leq l_t$ and $l_2\geq2$.
In particular, $\theta(l,l)\cong C_{2l}.$ Therefore, the problem of determining $ex(n, \theta(l_1,l_2,\ldots,l_t))$
generalizes the problem of determining $ex(n, C_{2l})$.
With the restriction $l_1=\cdots=l_t=l$,
Faudree and Simonovits \cite{FS} showed that a $\theta(l,l,\ldots,l)$-free graph on $n$ vertices has $O_{t,l}(n^{1+1/l})$ edges.
Recently, Bukh and Tait \cite{BT} improved their result.
With the restriction $t=3$, we obtain a \emph{theta graph} $\theta(p,q,r)$.
For the lower bound of $ex(n,\theta(4,4,4))$,
Verstra\"{e}te and Williford \cite{bag} proved a one whose order of magnitude is $n^{5/4}$.
In this paper, we focus on a class of theta graphs $\theta(1,2,r+1)$.
For convenience, we use $\theta_{r+1}$ to denote $\theta(1,2,r+1)$.
Obviously, $\chi(\theta_{r+1})=3$ and $\theta_{r+1}$ contains a color critical edge.

\begin{thm}\label{thm2}
$Ex_{sp}(n,\theta_{r+1})=\{T_{n,2}\}$ for $n\geq 10r$ if $r$ is odd and $n\geq 7r$ if $r$ is even.
\end{thm}

The study of consecutive cycles is an important topic in extremal graph theory.
From \cite{BO}, we know that if $G$ is a graph of order $n$ with $e(G)>e(T_{n,2})$, then $G$ contains
a cycle of length $t$ for every $t\leq \lfloor\frac{n+3}2\rfloor$.
In 2008, Nikiforov \cite{NI2} studied the spectral condition for the existence of consecutive cycles.
He proved that if $G$ is a graph of order $n$ with $\rho(G)>\rho(T_{n,2})$,
then $G$ contains a cycle of length $t$ for every $t\leq \frac{n}{320}$.
Moreover, Nikiforov proposed the following problem.

\begin{prob}\label{que1}(Nikiforov, \cite{NI2})
Determine the maximum constant $c$ such that for all positive $\varepsilon<c$ and large enough $n$,
every graph $G$ of order $n$ with $\rho(G)>\rho(T_{n,2})$
contains a cycle of length $t$ for every $t\leq (c-\varepsilon)n$.
\end{prob}

Nikiforov's result implies that $c>\frac1 {320}$.
Moreover, Nikiforov \cite{NI2} constructed a graph $S_{n,k}$,
which is the join of a complete graph of order $k=\lceil\frac{(3-\sqrt{5})n}4\rceil$
with an empty graph of order $n-k$. One can see $\rho(S_{n,k})>\frac n2\geq \rho(T_{n,2})$, but $S_{n,k}$ contains no cycle of length
longer than $2k$ (If $S_{n,k}$ contains a cycle $C$ of length at least $2k+1$,
then $C$ contains at least $k+1$ vertices out of the $k$-clique.
Since any two of these $k+1$ vertices can not be consecutive in $C$,
$C$ contains at least $k+1$ vertices of the $k$-clique, a contradiction).
This implies that $c\leq\frac{3-\sqrt{5}}2<\frac1{2.5}.$
Recently, Ning and Peng \cite{NP} improved Nikiforov's result by showing $c>\frac1 {160}$.
By Theorem \ref{thm2}, for every even $r$ and every graph $G$ of order $n\geq 7r$,
if $\rho(G)>\rho(T_{n,2})$ then $G$ contains a $\theta_{r+1}$.
It follows that $G$ contains $\theta_{3}$,$\theta_{5}$,\ldots,$\theta_{r+1}$.
Note that $\theta_{t+1}$ contains $C_3$, $C_{t+2}$ and $C_{t+3}$ for each $t=2,4,\ldots,r$.
Thus, we get the following result.

\begin{cor}\label{cor2}
Let $n$ be an arbitrary positive integer and $G$ be a graph of order $n$ with $\rho(G)>\rho(T_{n,2})$.
Then $G$ contains a cycle of length $t$ for every $t\leq \frac{n}{7}$,
unless $G\cong T_{n,2}$.
\end{cor}

Corollary \ref{cor2} indicates that $c>\frac17$ without the restriction ``sufficiently large $n$".
Inspired by Nikiforov's problem,
one may ask Problem \ref{que1} by removing the condition ``large enough $n$".
Furthermore, we have the following problem.

\begin{prob}
Determine the maximum constant $c'$ such that for any positive integer $n$,
every graph $G$ of order $n$ with $\rho(G)>\rho(T_{n,2})$
contains  a copy of $\theta_{r+1}$ for every $r\leq c'n$.
\end{prob}

Theorem \ref{thm2} implies that $c'\geq\frac 1{10}.$ Moreover, combining Theorems \ref{thm1} and \ref{thm2} with Fact \ref{fact1},
we essentially determine the Tur\'{a}n numbers of $B_{r+1}$ and $\theta_{r+1}$.

\begin{cor}\label{corollary1}
Let $r$ be a positive integer. Then we have the following statements.

\vspace{1mm}
\noindent (i)
$ex(n,B_{r+1})=\lfloor\frac{n^2}{4}\rfloor$ for $n\geq \frac{13}{2}r$;

\vspace{1mm}
\noindent (ii)
$ex(n,\theta_{r+1})=\lfloor\frac{n^2}{4}\rfloor$ for $n\geq 10r$ if $r$ is odd and $n\geq 7r$ if $r$ is even.
\end{cor}

\section{Preliminaries}

All graphs considered here are undirected, loopless and simple.
Let $V(G)$ and $E(G)$ be the \emph{vertex set} and \emph{edge set} of a graph $G$, respectively.
As usual, we use $e(G)$ to denote the size of $G$.
For a vertex $u\in V(G)$ and a subgraph $H\subseteq G$ (possibly $u\notin V(H)$),
we use $N_H(u)$ to denote the set of neighbors of $u$ in $V(H)$ and $d_H(u)$ to denote $|N_H(u)|$.
If $uv\in E(G)$, then we write $u\sim v$.
The subgraph of $ G $ induced by a vertex subset $S$ is denoted by $G[S]$.
%We also use $G-S$ to denote the subgraph of $G$ induced by $V(G)\setminus S$ and $G-u$ to denote $G[V(G)\setminus\{u\}]$.
For two disjoint subsets $A,B\subseteq V(G)$, let $E(A)$ be the set of edges with both endpoints in $A$
and $E(A,B)$ be the set of edges with one endpoint in $A$ and the other in $B$. Furthermore, we denote
$e(A)=|E(A)|$ and $e(A,B)=|E(A,B)|$.

%Let $K_r^+(s_1,s_2,\ldots,s_r)$ denote the complete $r$-partite graph with parts of size $s_1\geq2$, $s_2, \ldots, s_r$,
%with an edge added to the first part.
%
%\begin{lem}[Nikiforov, \cite{ESB1}]\label{lem0}
%Let $r\geq2$, $c=r^{-(2r+9)(r+1)}$, $n\geq e^{\frac 2 c}$, and $G$ be a graph of order $n$.
%If $\rho(G)>\rho(T_{n,k})$, then $G$ contains a copy of $K_r^+(\lfloor c\ln n\rfloor,\lfloor c\ln n\rfloor,\ldots,\lfloor c\ln n\rfloor).$
%\end{lem}

The following result is known as Erd\H{o}s-Gallai theorem.
\begin{lem}[Erd\H{o}s, Gallai, \cite{ER}]\label{lem2}
Let $G$ be a $P_{r+2}$-free graph of order $n$. Then $e(G)\leq \frac{r}{2}n$,
and equality holds if and only if $G$ is a union of disjoint copies of $K_{r+1}$.
\end{lem}

\begin{lem}[\cite{ZLG}]\label{lem1}
Let $G =<X, Y>$ be a bipartite graph, where $|X|\geq r$ and $|Y|\geq r-1\geq1$.
If $G$ does not contain a copy of $P_{2r+1}$ with both endpoints in $X$, then
$$e(G)\leq (r-1)|X| + r|Y |-r(r-1).$$
Equality holds if and only if $G\cong K_{|X|,|Y|}$, where $|X| =r$ or $|Y| =r-1$.
\end{lem}

Assume that $H\in\{B_{r+1},\theta_{r+1}\}$ and $G$ is an extremal graph with respect to $ex_{sp}(n,H)$.
Let $X=(x_1,\ldots,x_n)^t$ be the Perron vector of $G$ and $u^\star\in V(G)$ with $x_{u^\star}=\max\{x_i:i=1,\ldots,n\}$.
Let $A=N_G(u^\star)$, $B=V(G)\backslash{(A\cup \{u^\star\})}$ and $$\gamma(u^\star)=|A|+2e(A)+e(A,B).$$
Then we have the following lemma.

\begin{lem}\label{lem3}
If $e(A)=0,$ then $|A|\geq\lceil\frac{n}{2}\rceil$, $\gamma(u^\star)\geq\lfloor\frac{n^2}{4}\rfloor$ and $G\cong T_{n,2}$
\end{lem}

\begin{proof}
Since $T_{n,2}$ is $H$-free and $G$ attains the maximum spectral radius, we have
$$\rho(G)\geq\rho(T_{n,2})=\sqrt{\lfloor\frac n2\rfloor\lceil\frac n 2\rceil}=\sqrt{\lfloor\frac{n^2}{4}\rfloor}>\frac{n-1}{2}.$$
Note that $\rho(G)x_{u^\star}=\sum_{u\in A}x_u\leq |A|x_{u^\star}$.
Then $|A|\geq \rho(G)>\frac{n-1}{2},$
which implies that
\begin{eqnarray}\label{eq1}
|A|\geq \lceil\frac{n}{2}\rceil.
\end{eqnarray}
By using eigen-equation again, we have
\begin{eqnarray*}
\rho^2(G)x_{u^\star}&=&\rho(G)\sum_{u\sim u^\star}x_u=\sum_{u\sim u^\star}\sum_{v\sim u}x_v\\
&=&|A|x_{u^\star}+\sum_{u\in A}d_A(u)x_u+\sum_{v\in B}d_A(v)x_v\\
&\leq&|A|x_{u^\star}+2e(A)x_{u^\star}+e(A,B)x_{u^\star}\\
&=&\gamma(u^\star)x_{u^\star}.
\end{eqnarray*}
It follows that
\begin{eqnarray}\label{eq2}
\gamma(u^\star)\geq \rho^2(G)\geq\lfloor\frac{n^2}{4}\rfloor.
\end{eqnarray}
Now if $e(A)=0$, then
\begin{eqnarray}\label{eq3}
\gamma(u^\star)=|A|+e(A,B)\leq |A|+|A||B|=|A|(|B|+1)\leq\lfloor\frac{n^2}{4}\rfloor,
\end{eqnarray}
since $|A|+(|B|+1)=n$. Combining with (\ref{eq2}), we have $\gamma(u^\star)=\lfloor\frac{n^2}{4}\rfloor.$
Consequently, all the inequalities in (\ref{eq3}) become equalities.
Therefore, $e(A,B)=|A||B|$. Moreover, combining with (\ref{eq1}) we have $|A|=\lceil\frac{n}{2}\rceil$
and $|B|+1=\lfloor\frac n2\rfloor$.
This implies that $G$ contains $T_{n,2}$ as a spanning subgraph.
Furthermore, $e(B)=0$, since $G$ is either $B_{r+1}$-free or $\theta_{r+1}$-free.
Thus, $G\cong T_{n,2}$, as desired.
\end{proof}

%\hfill{\rule{4pt}{8pt}}

\section{Spectral extrema for books}\label{se3}

In this section, we give the proof of Theorem \ref{thm1}.
Suppose that $G$ is an extremal graph with respect to $ex_{sp}(n,B_{r+1})$.
Let $X=(x_1,\ldots,x_n)^t$ be the Perron vector of $G$ and $u^\star$ be a vertex with $x_{u^\star}=\max\{x_i:i=1,\ldots,n\}.$
Since $T_{n,2}$ has no $B_{r+1}$, for $n\geq\frac{13} 2r$ we have
\begin{eqnarray}\label{eq-1}
\rho(G)\geq\rho(T_{n,2})=\sqrt{\lfloor\frac{n^2}{4}\rfloor}>3r.
\end{eqnarray}

\noindent{\bf Proof of Theorem \ref{thm1}:}
\vspace{2mm}

Let $A=N_G(u^\star)$ and $B=V(G)\backslash{(A\cup \{u^\star\})}$.
By Lemma \ref{lem3}, if $e(A)=0$, then $G\cong T_{n,2}$ and the theorem holds.
In the following suppose that $e(A)\neq0$. We first give some claims.

\begin{claim}\label{cla0} $d_{A}(u)\leq r$ for each $u\in A$.
\end{claim}
\begin{proof} Since $G$ is $B_{r+1}$-free, the claim holds immediately.
\end{proof}

\begin{claim}\label{cla1} $d_B(u_1)+d_B(u_2)\leq |B|+r-1$ for any $u_1u_2\in E(A)$.
\end{claim}

\begin{proof}
By the way of contradiction, assume that $d_B(u_1)+d_B(u_2)\geq |B|+r$ for some $u_1u_2\in E(A)$.
Note that
\begin{eqnarray*}
|B|+r&\leq& d_B(u_1)+d_B(u_2)=|N_B(u_1)\cup N_B(u_2)|+|N_B(u_1)\cap N_B(u_2)|\\
&\leq& |B|+|N_B(u_1)\cap N_B(u_2)|.
\end{eqnarray*}
Then $|N_B(u_1)\cap N_B(u_2)|\geq r$, and $G$ contains a $B_{r+1}$, a contradiction.
\end{proof}

\begin{claim}\label{cla2}  $x_{u_1}+x_{u_2}\leq \frac{|B|+r+1}{\rho(G)-r}x_{u^\star}$ for any $u_1u_2\in E(A)$.
\end{claim}

\begin{proof}
Suppose that $x_u+x_v=\max_{uv\in E(A)}{(x_u+x_v)}$.
Now we only need to show that $x_u+x_v\leq \frac{|B|+r+1}{\rho(G)-r}x_{u^\star}$.
By Claim \ref{cla0}, we have
$$\rho(G)x_u=x_{u^\star}+\sum_{w\in N_A(u)}x_{w}+\sum_{w\in N_B(u)}x_{w}\leq x_{u^\star}+rx_{v}+d_B(u)x_{u^\star}$$
and
$$\rho(G)x_v=x_{u^\star}+\sum_{w\in N_A(v)}x_{w}+\sum_{w\in N_B(v)}x_{w}\leq x_{u^\star}+rx_{u}+d_B(v)x_{u^\star}.$$
Combining the above two inequalities and using Claim \ref{cla1}, we have
$$(\rho(G)-r)(x_u+x_v)\leq 2x_{u^\star}+(d_B(u)+d_B(v))x_{u^\star}\leq (|B|+r+1)x_{u^\star}.$$
Then the claim holds.
\end{proof}
We use $\overline{e(A,B)}$ to denote the number of edges missing in $E(A,B)$, that is, $\overline{e(A,B)}=|A||B|-e(A,B)$.
For the sake of simplicity, we use $\overline{d_B}(u)$ instead of $\overline{e(\{u\},B)}$, where $u\in A$. Now we have the following claim.

\begin{claim}\label{cla3}  $\overline{e(A,B)}\geq \frac{1}{r}e(A)(|B|-r+1)$.
\end{claim}

\begin{proof}
By Claim \ref{cla1}, $d_B(u_1)+d_B(u_2)\leq |B|+r-1$ for $u_1u_2\in E(A)$.
Then $$\overline{d_B}(u_1)+\overline{d_B}(u_2)\geq 2|B|-(|B|+r-1)=|B|-r+1.$$
For each $u\in A$,
$\overline{d_B}(u)$ appears at most $r$ times in $\sum_{u_1u_2\in E(A)}\left(\overline{d_B}(u_1)+\overline{d_B}(u_2)\right)$,
since $d_A(u)\leq r$.
Thus
\begin{eqnarray*}
\overline{e(A,B)}=\sum_{u\in A}\overline{d_B}(u)\geq\frac{1}{r}\sum_{u_1u_2\in E(A)}\left(\overline{d_B}(u_1)+\overline{d_B}(u_2)\right)\geq\frac{1}{r}e(A)(|B|-r+1),
\end{eqnarray*}
as required.
\end{proof}
\vspace{3mm}

Now we give the final proof of Theorem \ref{thm1}.
Since $\rho(G)x_{v}=\sum_{u\sim v}x_u,$ it follows that
\begin{eqnarray*}
\rho^2(G)x_{u^\star}&=&\rho(G)\sum_{u\sim u^\star}x_u=\sum_{u\sim u^\star}\sum_{v\sim u}x_v\\
&=&|A|x_{u^\star}+\sum_{u\in A}d_A(u)x_u+\sum_{v\in B}d_A(v)x_v\\
&=&|A|x_{u^\star}+\sum_{u_1u_2\in E(A)}(x_{u_1}+x_{u_2})+\sum_{v\in B}d_A(v)x_v\\
&\leq&|A|x_{u^\star}+e(A)\frac{|B|+r+1}{\rho(G)-r}x_{u^\star}+e(A,B)x_{u^\star}~~\mbox{by Claim \ref{cla2}}.
\end{eqnarray*}
Note that $e(A,B)=|A||B|-\overline{e(A,B)}$. Then by Claim \ref{cla3},
\begin{equation}\label{eq0}\begin{array}{ll}
\rho^2(G)&\leq |A|+e(A)\frac{|B|+r+1}{\rho(G)-r}+|A||B|-\frac{1}{r}e(A)(|B|-r+1)\\
&= |A|(|B|+1)+\left(\frac{|B|+r+1}{\rho(G)-r}-\frac{|B|-r+1}{r}\right)e(A).
\end{array}\end{equation}

If $\rho(G)>r\left(1+\frac{|B|+r+1}{|B|-r+1}\right)$, then $\frac{|B|+r+1}{\rho(G)-r}<\frac{|B|-r+1}{r}$ and thus
$\rho^2(G)<|A|(|B|+1)\leq \lfloor\frac{n^2}{4}\rfloor,$
which contradicts (\ref{eq-1}).

If
$\rho(G)=r\left(1+\frac{|B|+r+1}{|B|-r+1}\right)$, then $\rho^2(G)\leq|A|(|B|+1)\leq \lfloor\frac{n^2}{4}\rfloor.$
Combining (\ref{eq-1}), we have $\rho^2(G)=\lfloor\frac{n^2}{4}\rfloor.$
From Lemma \ref{lem3}, it leads to $|A|=\lceil\frac{n}{2}\rceil$ and $|B|+1=\lfloor\frac{n}{2}\rfloor$.
Furthermore, by $\rho(G)=r\left(1+\frac{|B|+r+1}{|B|-r+1}\right)=\frac {2r(|B|+1)}{|B|-r+1}$, then
\begin{eqnarray}\label{eq4}
\rho^2(G)
=\left(\frac{2r}{\lfloor\frac{n}{2}\rfloor-r}\right)^2{\lfloor\frac{n}{2}\rfloor}^2
\leq \left(\frac{2r}{\lfloor\frac{n}{2}\rfloor-r}\right)^2\lfloor\frac{n^2}{4}\rfloor
=\left(\frac{2r}{\lfloor\frac{n}{2}\rfloor-r}\right)^2\rho^2(G)\leq\rho^2(G),
\end{eqnarray}
since $n\geq \frac{13}2r$. Therefore, $\lfloor\frac{n}{2}\rfloor^2=\lfloor\frac{n^2}{4}\rfloor$, which implies that $n$ is even.
Consequently, $$\frac{2r}{\lfloor\frac{n}{2}\rfloor-r}=\frac{2r}{\frac{n}{2}-r}\leq \frac{8}{9}.$$ Thus, (\ref{eq4}) becomes a strict inequality,
also a contradiction.

If $\rho(G)<r\left(1+\frac{|B|+r+1}{|B|-r+1}\right)$, then $\frac{|B|+r+1}{\rho(G)-r}>\frac{|B|-r+1}{r}$. By (\ref{eq0}), we have
\begin{eqnarray*}
\rho^2(G)&\leq& |A|(|B|+1)+\left(\frac{|B|+r+1}{\rho(G)-r}-\frac{|B|-r+1}{r}\right)\frac{r}{2}|A|~~\mbox{by Claim \ref{cla0}}\\
&<&|A|(|B|+1)+\left(\frac{|B|+r+1}{2r}-\frac{|B|-r+1}{r}\right)\frac{r}{2}|A|~~\mbox{since $\rho(G)>3r$ by (\ref{eq-1})}\\
&=&\frac{3}{4}|A|(|B|+1+r)\leq\frac{3}{4}\lfloor\frac{(|A|+|B|+1+r)^2}{4}\rfloor\\
&=&\frac{3}{4}\lfloor\frac{(n+r)^2}{4}\rfloor~~\mbox{since $|A|+|B|+1=n$}\\
&\leq&\lfloor\frac{n^2}{4}\rfloor~~\mbox{as $n\geq \frac{13}{2}r$},
\end{eqnarray*}
a contradiction. This completes the proof.
\hfill{\rule{4pt}{8pt}}

\section{Spectral extrema for theta graphs}\label{se4}

In this section, we give the proof of Theorem \ref{thm2}.
Suppose that $G$ is an extremal graph with respect to $ex_{sp}(n,\theta_{r+1})$.
Let $X=(x_1,\ldots,x_n)^t$ be the Perron vector of $G$ and $u^\star\in V(G)$ with $x_{u^\star}=\max\{x_i:i=1,\ldots,n\}$.
Let $A=N_G(u^\star)$, $B=V(G)\setminus {(A\cup\{u^\star\})}$ and $$\gamma(u^\star)=|A|+2e(A)+e(A,B).$$

\noindent{\bf Proof of Theorem \ref{thm2}.}
We derive the proof by the following two cases.

\noindent{\bf Case 1. $r$ is odd.}
\vspace{1mm}

Now $n\geq10r$. If $r=1$, then $\theta_{r+1}\cong B_{r+1}$. By Theorem \ref{thm1},
$Ex_{sp}(n,\theta_2)=\{T_{n,2}\}$ for $n\geq7.$ In the following, we may assume that $r\geq3$.
We first give several claims.

\begin{claim}\label{cla4} $|B|\geq 2r$. Moreover, $e(H)\leq \frac{r}{2}|H|$ for each component $H$ of $G[A]$.
\end{claim}

\begin{proof}
Since $G$ is $\theta_{r+1}$-free, $G[A]$ is $P_{r+2}$-free.
By Lemma \ref{lem2}, $e(H)\leq \frac{r}{2}|H|$ for each component $H$ of $G[A]$.
It follows that $e(A)\leq\frac{r}{2}|A|$ and then
 $$\gamma(u^\star)= |A|+2e(A)+e(A,B)\leq (1+r+|B|)|A|=(1+r+|B|)(n-1-|B|).$$
Suppose that $|B|\leq 2r-1$. Since $n\geq10r$, $\gamma(u^\star)\leq 3r\cdot8r=24r^2<\lfloor\frac {n^2}4\rfloor,$
which contradicts (\ref{eq2}). Therefore, $|B|\geq 2r$.
\end{proof}

An \emph{$A$-path} of $G$ is a path consisting of edges in $E(A,B)$ and with both endpoints in $A$.
The length of a path $P$ is denoted by $l(P)$.

\begin{claim}\label{cla5} There exists an $A$-path of length $3r-1$ in $G$.
\end{claim}

\begin{proof}
By Lemma \ref{lem3}, we know that $|A|\geq \lceil\frac{n}{2}\rceil$.
By the way of contradiction, suppose that there exists no $A$-path of length $3r-1$.
Then by Lemma \ref{lem1},
$$e(A,B)\leq \frac{3r-3}{2}|A|+\frac{3r-1}{2}|B|-\frac{(3r-1)(3r-3)}{4}.$$
Combining with $e(A)\leq \frac{r}{2}|A|$, we have
\begin{eqnarray*}
\gamma(u^\star)&=& |A|+2e(A)+e(A,B)\leq \frac{5r-1}{2}|A|+\frac{3r-1}{2}|B|-\frac{(3r-1)(3r-3)}{4}\\
&=&\frac{5r-1}{2}(|A|+|B|+1)-r|B|-\frac{9r^2-2r+1}{4}\\
&\leq&\frac{(5r-1)n}{2}-2r^2-\frac{9r^2-2r+1}{4}~~\mbox{since $|B|\geq 2r$}\\
&\leq&\frac{10rn-17r^2}{4}<\lfloor\frac{n^2}{4}\rfloor~~\mbox{as $n\geq 10r$},
\end{eqnarray*}
a contradiction. So the claim holds.
\end{proof}

The set of internal vertices of a path $P$ is denoted by $V_{in}(P)$.
For two vertices $u,v\in V(P)$, the distance between $u$ and $v$ in $P$ is denoted by $d_P(u,v)$.
A vertex of a $\theta_{r+1}$ is said to be a \emph{head}, if it is of degree two and it belongs to a triangle of $\theta_{r+1}$.
The following two claims give structural properties of $G$.

\begin{claim}\label{cla6} Let $P$ be an $A$-path in $G$ and $u\in A\setminus V(P)$.
The following statements hold.

\vspace{1mm}
\noindent(i) If $l(P)\geq r+3$ and $u\sim u'$ for some $u'\in A\setminus V_{in}(P)$, then $\overline{e(\{u,u'\},B)}\geq 3$.

\vspace{1mm}
\noindent(ii) If $l(P)\geq 2r+2$ and $u\sim u'\sim u''$ for some $u',u''\in A\setminus V(P)$, then $\overline{e(\{u\},B)}\geq r+1$.
\end{claim}

\begin{proof} (i) Let $P=u_1v_1u_2v_2\cdots u_{\frac{r+3}{2}}v_{\frac{r+3}{2}}u_{\frac{r+5}{2}}$ be an $A$-path of length $r+3$,
where $u_i \in A$ and $v_i\in B.$
Pick two pairs of vertices $\{v_1,v_{1+\frac{r-1}{2}}\}$ and $\{v_2,v_{2+\frac{r-1}{2}}\}$.
Clearly, $d_P(v_i,v_{i+\frac{r-1}{2}})=r-1$ for $i\in\{1,2\}.$
If $e(\{u,u'\},\{v_i,v_{i+\frac{r-1}{2}}\})\geq 3$, then there are two independent edges, say $uv_i,u'v_{i+\frac{r-1}{2}}$, in $E(\{u,u'\},\{v_i,v_{i+\frac{r-1}{2}}\})$.
Therefore, $uv_iu_{i+1}v_{i+1}\cdots u_{i+\frac{r-1}{2}}v_{i+\frac{r-1}{2}}u'$ is an $A$-path of length $r+1$.
Combining with the triangle $u u^\star u'$,
we get a $\theta_{r+1}$ with its head $u^\star$, a contradiction.
Thus $e(\{u,u'\},\{v_i,v_{i+\frac{r-1}{2}}\})\leq 2$ and then $\overline{e(\{u,u'\},\{v_i,v_{i+\frac{r-1}{2}}\})}\geq 2$
for $i\in\{1,2\}.$

If $r\geq 5$, then $\{v_1,v_{1+\frac{r-1}{2}}\}\cap\{v_2,v_{2+\frac{r-1}{2}}\}=\varnothing$,
and thus $\overline{e(\{u,u'\},B)}\geq 4,$ as required.
Now consider $r=3$. It suffices to show $\overline{e(\{u,u'\},\{v_1,v_2,v_3\})}\geq 3$.
Note that $\overline{e(\{u,u'\},\{v_1,v_2\})}\geq 2$ and $\overline{e(\{u,u'\},\{v_2,v_3\})}\geq 2$.
If $\overline{e(\{u,u'\},\{v_1,v_2,v_3\})}\leq 2$,
then $\overline{e(\{u,u'\},\{v_2\})}=2$ and $\overline{e(\{u,u'\},\{v_1,v_3\})}=0$.
This implies that $e(\{u,u'\},\{v_1,v_3\})=4$.
Notice that $u'\notin V_{in}(P)$. Then, either $u'\notin \{u_1,u_2,u_3\}$ or $u'\notin \{u_2,u_3,u_4\}$.
Without loss of generality, assume that $u'\notin \{u_1,u_2,u_3\}$, then $u'u^\star u_1v_1u$ is a path of length 4.
Together with the triangle $u'v_3u$, we get a $\theta_4$ with its head $v_3$, a contradiction.
Therefore, $\overline{e(\{u,u'\},B)}\geq 3$.

\vspace{1mm}
(ii) Let $P=u_1v_1u_2v_2\cdots u_{r+1}v_{r+1}u_{r+2}$ be an $A$-path of length $2r+2$, where $u_i \in A$ and $v_i\in B.$
If $u\sim v_i$ for some $i\in \{1,2,\ldots,r+1\}$, then select a vertex $v_j\in V(P)$ such that $d_P(v_i,v_j)=r-3$.
Let $P_{v_i\rightarrow v_j}$ be a subpath of $P$ with endpoints $v_i$ and $v_j$.
Then $u'uP_{v_i\rightarrow v_j}u_{j+1}u^\star$ is a path of length $r+1$.
Combining with the triangle $u'u''u^\star$, we get a $\theta_{r+1}$ with its head $u''$, a contradiction.
It follows that $N_G(u)\cap \{v_1,\ldots,v_{r+1}\}=\emptyset$ and then the result holds.
\end{proof}

\begin{claim}\label{cla7} Let $P$ be an $A$-path in $G$ starting from some vertex $u\in A$. If
 $l(P)\geq r-1$ and $u\sim u'\sim u''$ for some $u',u''\in A$, then either $u'\in V(P)$ or $u''\in V(P)$.
\end{claim}

\begin{proof} Let $P_{u\rightarrow u'''}$ be an $A$-path of length $r-1$ with endpoints $u$ and $u'''$.
If $u',u''\notin V(P)$, then $u'P_{u\rightarrow u'''}u^\star$ is a path of length $r+1$. Combining with the triangle $u'u''u^\star$,
we get a $\theta_{r+1}$ with its head $u''$, a contradiction.
So the claim holds.\end{proof}

A component $H$ of $G[A]$ is said to be \textbf{\emph{good}}, if $2e(H)+e(V(H),B)<|H||B|$.
Clearly, $H$ is not good if $H$ is an isolated vertex in $G[A]$.
The following claim gives a characterization of good components in $G[A]$.

\begin{claim}\label{cla8} Let $H$ be a component of $G[A]$ and $P$ be an $A$-path of length $3r-1$ in $G$.
The following statements hold.

\vspace{1mm}
\noindent (i) If $|H|=2$ and $V(H)\nsubseteq V(P)$, then $H$ is good.

\vspace{1mm}
\noindent (ii) If $|H|\geq 3$ and $V^\star(H)\nsubseteq V(P)$, where $V^\star(H)=\{v\in V(H):d_H(v)\geq 2\}$, then $H$ is good.
\end{claim}

\begin{proof}
(i) Let $V(H)=\{u,u'\}$. Since $V(H)\nsubseteq V(P)$,
without loss of generality, we may assume that $u\notin V(P)$.
If $r\geq 5$, then $l(P)=3r-1\geq 2r+4$.
Therefore, regardless of whether $u'\in V(P)$ or not,
there exists an $A$-subpath $P_1$ of $P$ with $l(P_1)=r+3$ and $u'\notin V_{in}(P_1)$.
By Claim \ref{cla6} (i), we have $\overline{e(\{u,u'\},B)}\geq 3$. It follows that
$$2e(H)+e(V(H),B)\leq 2+|H||B|-3<|H||B|,$$
and then $H$ is a good component.

If $r=3$, then $l(P)=8$. Let $P=u_1v_1\cdots u_4v_4u_5,$ where $u_i\in A$ and $v_i\in B$.
If $u'\neq u_3$ (where $u_3$ is the central vertex of $P$),
then there exists an $A$-subpath $P_2$ of $P$ with $l(P_2)=r+3$ and $u'\notin V_{in}(P_2)$.
Similar as above, $H$ is good. If $u'=u_3$, then $\overline{e(\{u,u_3\},\{v_1,v_2\})}\geq 2$
(otherwise, there are two independent edges, say $uv_1, u_3v_2$, in $E(\{u,u_3\},\{v_1,v_2\})$.
Thus $uv_1u_2v_2u_3$ is an $A$-path of length 4 and then there exists a $\theta_4$ with its head $u^\star$).
Similarly, $\overline{e(\{u,u_3\},\{v_3,v_4\})}\geq 2$.
It follows that
$$2e(H)+e(V(H),B)\leq 2+|H||B|-4<|H||B|,$$
as desired.

\vspace{1mm}
(ii) We first show $V(H)\cap V(P)=\varnothing.$
Suppose to the contrary,
then there exists a path $P^\star\subseteq H$ with one endpoint in $V(P)$
and the other in $V^\star(H)\setminus V(P)$, since $V^\star(H)\nsubseteq V(P)$.
Furthermore, we can find an edge $uu'\in E(P^\star)$ with $u\in V(P)$ and $u'\in V^\star(H)\setminus V(P)$.
By the definition of $V^\star(H)$, we have $d_H(u')\geq2$. Let $u''\in N_H(u')\setminus \{u\}$.
Note that $l(P)\geq 3r-1$. Then there exists an $A$-subpath $P_3$ of $P$, starting from $u$, such that $l(P_3)\geq r-1$. Since $u'\notin V(P)$ and then $u'\notin V(P_3)$,
by Claim \ref{cla7}, we have $u''\in V(P_3)\subseteq V(P).$
Note that $l(P)\geq3r-1$ and $u,u''\in V(P)$.
Then there exists either an $A$-subpath $P_4$ of $P$ starting from $u$ with $l(P_4)\geq r-1$ and $u''\notin V(P_4)$,
or an $A$-subpath $P_5$ of $P$ starting from $u''$ with $l(P_5)\geq r-1$ and $u\notin V(P_5)$.
By Claim \ref{cla7}, in both cases we have $u'\in V(P)$, a contradiction.
Therefore, $V(H)\cap V(P)=\varnothing$.

Note that $l(P)=3r-1\geq 2r+2$.
If $H\cong K_{1,s}$ ($s\geq 2$), then for each pendent vertex $u\in V(H)$ there exist $u',u''\in V(H)$ such that $u\sim u'\sim u''$.
By Claim \ref{cla6} (ii), we have $\overline{e(\{u\},B)}\geq r+1$.
Hence, $$2e(H)+e(V(H),B)\leq 2s+|H||B|-(r+1)s<|H||B|,$$ as required.
If $H$ is not a star, then for each vertex $u\in V(H)$ there exist $u',u''\in V(H)$ such that $u\sim u'\sim u''$.
Again by Claim \ref{cla6} (ii), we have $\overline{e(\{u\},B)}\geq r+1$.
Combining with Claim \ref{cla4}, we have $$2e(H)+e(V(H),B)\leq r|H|+|H||B|-(r+1)|H|<|H||B|,$$ as required.
\end{proof}

For a subset $A'\subseteq A$, an $A'$-path of $G$ is a path consisting of edges in $E(A',B)$ and with both endpoints in $A'$.
Claim \ref{cla5} states that there exists an $A$-path of length $3r-1$ in $G$.
The following claim gives a stronger characterization.

\begin{claim}\label{cla9} Let $P$ be an $A$-path of length $3r-1$ in $G$ and $A'=A\setminus V(P)$.
Then there exists an $A'$-path $P'$ of length $3r-1$.
\end{claim}

\begin{proof} Note that $|B|\geq 2r$, $|A|\geq \lceil\frac{n}{2}\rceil\geq 5r$ and
$|A\setminus A'|=|V(P)\cap A|=\frac{3r+1}{2}$.
Then $|B|>\frac{3r-3}{2}$ and $|A'|=|A|-\frac{3r+1}{2}>\frac{3r-1}{2}.$
If $G$ does not contain an $A'$-path of length $3r-1$, then by Lemma \ref{lem1},
$$e(A',B)\leq \frac{3r-3}{2}|A'|+\frac{3r-1}{2}|B|-\frac{(3r-1)(3r-3)}{4}.$$
Moreover, $$e(A\setminus A',B)\leq|A\setminus A'||B|=\frac{3r+1}{2}|B|.$$
By Claim \ref{cla4}, we have $2e(A)\leq r|A|$.
Combining these inequalities,
\begin{eqnarray*}
\gamma(u^\star)&=&|A|+2e(A)+e(A,B)\\
&\leq&(1+r)|A|+\frac{3r-3}{2}|A'|+3r|B|-\frac{(3r-1)(3r-3)}{4}\\
&=&(1+r)|A|+\frac{3r-3}{2}\left(|A|-\frac{3r+1}{2}\right)+3r|B|-\frac{(3r-1)(3r-3)}{4}\\
&=&\frac{5r-1}{2}|A|+3r|B|-\frac{3r(3r-3)}{2}\\
&=&3r(|A|+|B|+1)-\frac{r+1}{2}|A|-\frac{3r(3r-1)}{2}\\
&\leq&3rn-\frac{r+1}{4}n-3r^2~~\mbox{since $|A|\geq\frac{n}{2}$ and $3r-1\geq2r$}\\
&<&\frac{11}{4}rn-3r^2<\lfloor\frac{n^2}{4}\rfloor~~\mbox{as $n\geq 10r$},
\end{eqnarray*}
a contradiction. So the claim holds.
\end{proof}

In the rest, we only need to show that $e(A)=0$, and then by Lemma \ref{lem3}, we complete the proof.
By the way of contradiction, suppose that $e(A)\neq 0.$
Then $G[A]$ contains nontrivial components.
Let $A_0=\{u\in A: d_A(u)=0\}$ and $A_1=A\backslash A_0$.
If all nontrivial components of $G[A]$ are good components, then
$$2e(A_1)+e(A_1,B)<|A_1||B|.$$
It follows that
\begin{eqnarray*}
\gamma(u^\star)&=&|A|+2e(A_1)+e(A_1,B)+e(A_0,B)\\
&<&|A|+|A_1||B|+|A_0||B|\\
&=&|A|(|B|+1)\leq \lfloor\frac{n^2}{4}\rfloor,
\end{eqnarray*}
a contradiction.
Thus, $G[A]$ contains a nontrivial component $H$ which is not good.
By Claim \ref{cla9}, there exist two $A$-paths, $P$ and $P'$,
such that $l(P)=l(P')=3r-1$ and $V(P)\cap V(P')\cap A=\varnothing.$
Furthermore, by Claim \ref{cla8},
if $|H|=2$, then $V(H)\subseteq V(P)$ and $V(H)\subseteq V(P')$, and
if $|H|\geq3$, then $V^\star(H)\subseteq V(P)$ and $V^\star(H)\subseteq V(P')$.
Both cases contradict the fact $V(P)\cap V(P')\cap A=\varnothing$.
This completes the proof of Case 1.
\hfill{\rule{4pt}{8pt}}

\vspace{2.5mm}

\noindent{\bf Case 2. $r$ is even.}
\vspace{1.5mm}

Now, $n\geq7r$. We first give four claims.
\begin{claim}\label{cla10}
If $P$ is an $A$-path of length $r$ starting from $u_1$, then $N_A(u_1)\subseteq V(P)$.
\end{claim}

\begin{proof}
Let $P=u_1v_1\cdots u_{\frac{r}{2}}v_{\frac{r}{2}}u_{\frac{r}{2}+1}$, where $u_i\in A$ and $v_i\in B$.
Then $P^\star=P+u_{\frac{r}{2}+1}u^\star$ is a path of length $r+1$.
If there exists a vertex $u\in N_A(u_1)\setminus V(P)$, then $P^\star$
and the triangle $u_1uu^\star$
consist a $\theta_{r+1}$ with its head $u$, a contradiction.
\end{proof}

\begin{claim}\label{cla11}
$|B|\geq r$ and $G$ contains an $A$-path of length $2r$.
\end{claim}

\begin{proof}
We first show that $|B|\geq r$. Suppose to the contrary that $|B|\leq r-1$.
Since $G$ is $\theta_{r+1}$-free, $G[A]$ is $P_{r+2}$-free.
Then by Lemma \ref{lem2}, $2e(A)\leq r|A|$. It follows that
\begin{eqnarray*}
\gamma(u^\star)&=&|A|+2e(A)+e(A,B)\leq (1+r)|A|+|A||B|\\
&=&(1+r+|B|)(n-1-|B|)\\
&\leq& 2r(n-r)~~\mbox{since $|B|\leq r-1$}\\
&<& \lfloor\frac{n^2}{4}\rfloor~~\mbox{as $n\geq 7r$ and $r\geq 2$},
\end{eqnarray*}
a contradiction. Then $|B|\geq r$, as required.

Now, suppose that $G$ does not contain an $A$-path of length $2r$.
Then by Lemma \ref{lem1}, $$e(A,B)\leq (r-1)|A|+r|B|-r(r-1).$$
Furthermore,
\begin{eqnarray*}
\gamma(u^\star)&=&|A|+2e(A)+e(A,B)\leq (1+r)|A|+e(A,B)\\
&\leq& 2r|A|+r|B|-r(r-1)\\
&=&2r(|A|+|B|+1)-r(|B|+r+1)\\
&\leq& 2rn-r(2r+1)~~\mbox{since $|B|\geq r$}\\
&<& \lfloor\frac{n^2}{4}\rfloor~~\mbox{as $n\geq 7r$},
\end{eqnarray*}
a contradiction. So the claim holds.
\end{proof}

\begin{claim}\label{cla12}
Let $P$ be an $A$-path of length $2r$ in $G$ and $A'=A\setminus V(P)$.
Then there exists an $A'$-path $P'$ of length $2r$.
\end{claim}

\begin{proof} Let $P=u_1v_1\cdots u_rv_ru_{r+1}$, where $u_i\in A$ and $v_i\in B$.
Then $|A\setminus A'|=|V(P)\cap A|=r+1$ and thus
$|A'|=|A|-(r+1)>r$, since $|A|\geq\lceil\frac{n}{2}\rceil$ and $n\geq 7r$.
If $G$ does not contain an $A'$-path of length $2r$,
then by Lemma \ref{lem1},
\begin{eqnarray}\label{eq7}
e(A',B)\leq (r-1)|A'|+r|B|-r(r-1).
\end{eqnarray}

We first show that $e(A',A\setminus A')=0$.
Note that $A\setminus A'=\{u_1,\ldots, u_{r+1}\}$. Suppose that $uu_i\in E(A)$ for some $u\in A'$ and $u_i\in A\setminus A'$.
Without loss of generality, we may assume that $i\leq \frac{r}{2}+1.$
Then $u_iv_i\cdots u_{i+\frac{r}{2}-1}v_{i+\frac{r}{2}-1}u_{i+\frac{r}{2}}u^\star$ is a path of length $r+1$.
Together with the triangle $u_iuu^\star$, we get a $\theta_{r+1}$ with its head $u$, a contradiction.
Therefore, $e(A',A\setminus A')=0$.

Next, we show that $e(A\setminus A')\leq \frac{r^2}{8}.$
Let $$V_{11}=\{u_i:1\leq i\leq \frac{r}{2}+1\},~~~V_{12}=\{u_i:\frac{r}{2}+1\leq i\leq r+1\}.$$
If $u_iu_j\in E(G)$ for some $i,j$ with $1\leq i<j\leq \frac{r}{2}+1$,
then $P^\star=u_jv_j\cdots u_{j+\frac{r}{2}-1}v_{j+\frac{r}{2}-1}u_{j+\frac{r}{2}}$
is an $A$-path of length $r$ starting from $u_j$.
By Claim \ref{cla10}, $N_A(u_j)\subseteq V(P^\star)$ and then $u_i\in V(P^\star)$, a contradiction.
So, $e(V_{11})=0$.
Similarly, $e(V_{12})=0$. Hence, $G[A\setminus A']$ is a bipartite graph with an isolated vertex $u_{\frac{r}{2}+1}$.
Furthermore, for each $i\leq \frac{r}{2}$, $u_iv_i\cdots u_{i+\frac{r}{2}-1}v_{i+\frac{r}{2}-1}u_{i+\frac{r}{2}}$
is an $A$-path of length $r$ starting from $u_i$.
Again by Claim \ref{cla10},
if $u_iu_j\in E(G)$ for some $j\geq i+\frac{r}{2}+1$, we will get a contradiction.
Now one can see that $N_{A\setminus A'}(u_i)\subseteq \{u_j:\frac{r}{2}+2\leq j\leq \frac{r}{2}+i\}$ for each $i\leq \frac{r}{2}$ (where $N_{A\setminus A'}(u_1)=\varnothing$). Thus, $d_{A\setminus A'}(u_i)\leq i-1$ for each $i\leq \frac{r}{2}$.
It follows that
$$e(A\setminus A')=\sum_{i=1}^{\frac r2}d_{A\setminus A'}(u_i)\leq \sum_{i=1}^{\frac r2}(i-1)=\frac{r(r-2)}{8}<\frac{r^2}{8}.$$

Note that $G[A']$ is $P_{r+2}$-free. Then by Lemma \ref{lem2}, $2e(A')\leq r|A'|$. Thus
\begin{eqnarray}\label{eq8}
2e(A)=2e(A')+2e(A\setminus A')< r|A'|+\frac{r^2}{4}.
\end{eqnarray}
Note that \begin{eqnarray}\label{eq9}
e(A\setminus A',B)\leq |A\setminus A'||B|=(r+1)|B|.
\end{eqnarray}
Then combining (\ref{eq7}), (\ref{eq8}) with (\ref{eq9}), we have
\begin{eqnarray*}
\gamma(u^\star)&=&|A|+2e(A)+e(A',B)+e(A\setminus A',B)\\
&\leq&|A|+(2r-1)|A'|+(2r+1)|B|-\left(\frac{3}{4}r^2-r\right)\\
&=&|A|+(2r-1)(|A|-(r+1))+(2r+1)|B|-\left(\frac{3}{4}r^2-r\right)\\
&=&2r(|A|+|B|+1)+(|B|+1)-\left(\frac{11}{4}r^2+2r\right)\\
&\leq &\left(2r+\frac{1}{2}\right)n-\left(\frac{11}{4}r^2+2r\right)~~\mbox{since $|B|+1=n-|A|\leq \frac n2$}\\
&<& \lfloor\frac{n^2}{4}\rfloor~~\mbox{as $n\geq 7r$ and $r\geq 2$},
\end{eqnarray*}
a contradiction. So the claim holds.
\end{proof}

\begin{claim}\label{cla13}
Let $H$ be a nontrivial component of $G[A]$. Then $2e(H)+e(V(H),B)\leq |H||B|$,
and if equality holds then $H\cong K_{r+1}$.
\end{claim}

\begin{proof}
By Claim \ref{cla11}, there exists an $A$-path $P$ with $l(P)=2r$.
We first show that either $V(H)\subseteq V(P)$ or $V(H)\cap V(P)=\varnothing$.
By the way of contradiction, suppose that $u\in V(H)\cap V(P)$ and $u'\in V(H)\setminus V(P)$ such that $u\sim u'$.
Since $l(P)=2r$, there exists an $A$-subpath $P^\star$ of $P$ of length $r$ starting from $u$. By Claim \ref{cla10},
we have $u'\in V(P^\star)\subseteq V(P)$, a contradiction. Hence $V(H)\subseteq V(P)$ or $V(H)\cap V(P)=\varnothing$.

Now by Claim \ref{cla12}, there exists another $A$-path $P'$ of length $2r$ with $V(P)\cap V(P')\cap A=\varnothing.$
Similarly, we have either $V(H)\subseteq V(P')$ or $V(H)\cap V(P')=\varnothing$.
It follows that either $V(H)\cap V(P)=\varnothing$ or $V(H)\cap V(P')=\varnothing$.
Without loss of generality, assume that $V(H)\cap V(P)=\varnothing$.

Let $P=u_1v_1\cdots u_rv_ru_{r+1}$, where $u_i\in A$ and $v_i\in B$.
We shall show that $N_G(u)\cap \{v_1,\ldots,v_r\}=\varnothing$ for each $u\in V(H)$.
Suppose to the contrary that $u\sim v_i$ for some $u\in V(H)$ and some $i\in\{1,\ldots,r\}$.
Combining the edge $uv_i$ with the subpath of $P$ of length $r-1$ starting from $v_i$,
we can get an $A$-path of length $r$ starting from $u$.
By Claim \ref{cla10}, we have $N_A(u)\subseteq V(P)$, which contradicts $V(H)\cap V(P)=\varnothing$.

Now, we have $d_B(u)\leq |B|-r$ for each $u\in V(H)$. Thus $e(V(H),B)\leq |H|(|B|-r).$
Since $H$ is $P_{r+2}$-free, by Lemma \ref{lem2} we have $2e(H)\leq r|H|$, with equality if and only if $H\cong K_{r+1}$.
It follows that $2e(H)+e(V(H),B)\leq |H||B|$, and if equality holds then $H\cong K_{r+1}$.
\end{proof}

By Lemma \ref{lem3}, in the remaining, we only need to show that $e(A)=0$.
Claim \ref{cla13} implies that $2e(H)+e(V(H),B)\leq |H||B|$ for any component $H$ of $G[A]$.
Furthermore, if $G[A]$ contains a good component, then we have $2e(A)+e(A,B)<|A||B|$.
It follows that $$\gamma(u^\star)=|A|+2e(A)+e(A,B)<|A|(|B|+1)\leq \lfloor\frac{n^2}{4}\rfloor,$$
a contradiction. So we may assume that all the components of $G[A]$ are not good,
that is, \begin{eqnarray}\label{eq10}
2e(H)+e(V(H),B)=|H||B|
\end{eqnarray} for any component $H$ of $G[A]$.
This implies that $$\gamma(u^\star)=|A|+2e(A)+e(A,B)=|A|+|A||B|=|A|(|B|+1)\leq \lfloor\frac{n^2}{4}\rfloor,$$
Combining with Lemma \ref{lem3}, we have
$\gamma(u^\star)=\lfloor\frac{n^2}{4}\rfloor$, $|A|=\lceil\frac n2\rceil$ and $$|B|=n-1-|A|>r+1.$$

Now suppose that $e(A)\neq0$ and let $H$ be a nontrivial component of $G[A]$.
Then by (\ref{eq10}) and Claim \ref{cla13}, we have $H\cong K_{r+1}$.
Note that $G$ is $\theta_{r+1}$-free.
One can see that $d_{H}(w)\leq 1$ for each $w\in B$.
Therefore, $$2e(H)+e(V(H),B)\leq 2e(K_{r+1})+|B|=r(r+1)+|B|<r|B|+|B|=|H||B|,$$
since $r+1<|B|$.
This contradicts (\ref{eq10}). So $e(A)=0$. The proof of Case 2 is completed.
\hfill{\rule{4pt}{8pt}}

\end{document}